\documentclass[12pt]{amsart}
\usepackage{eurosym}
\usepackage{amscd,amsmath,amssymb,enumerate,amsfonts,mathrsfs,txfonts}
\usepackage{comment}
\usepackage{hyperref}
\usepackage{xcolor}
\usepackage{tikz,pgfplots,pgf}
\usepackage{vmargin}
\usepackage[all]{xy}

\tikzset{node distance=3cm, auto}
\usetikzlibrary{calc} 
\usetikzlibrary{trees} 
\setpapersize{A4}
\setmargins{2.5cm}      
{1.5cm}                        
{16.5cm}                      
{23.42cm}                   
{10pt}                          
{1cm}                           
{0pt}                          
{2cm}
\newtheorem{theorem}{Theorem}[section]
\theoremstyle{plain}

\newtheorem{corollary}[theorem]{Corollary}

\newtheorem{definition}[theorem]{Definition}
\newtheorem{example}[theorem]{Example}

\newtheorem{lemma}[theorem]{Lemma}

\newtheorem{proposition}[theorem]{Proposition}

\numberwithin{equation}{section}

\begin{document}
\title[On quotients of ideals of weighted holomorphic mappings]{On quotients of
ideals of weighted holomorphic mappings}
\author[A. Belacel]{A. Belacel}
\address[Amar Belacel]{Laboratory of Pure and Applied Mathematics (LPAM),
University of Laghouat, Laghouat, Algeria.}
\email{amarbelacel@yahoo.fr}
\author[A. Bougoutaia]{A. Bougoutaia}
\address[Amar Bougoutaia]{Laboratory of Pure and Applied Mathematics (LPAM),
University of Laghouat, Laghouat, Algeria.}
\email{amarbou28@gmail.com}
\author[P. Rueda]{P. Rueda}
\address[Pilar Rueda]{Departamento de An\'alisis Matem\'{a}tico, Universitat de Val\`encia, C/ Dr. Moliner 50, 46100  Burjassot (Valencia), Spain}
\email{pilar.rueda@uv.es}
\subjclass[2020]{46E15, 46B40, 46G20, 47L20}
\keywords{operator ideals, weighted Banach spaces, left-hand quotients, holomorphic mappings}

\begin{abstract}
We explore the procedure given by left-hand quotients in the context  of weighted holomorphic  ideals. On the one hand, we show that this procedure does  not generate new ideals other than the  ideal of weighted holomorphic mappings when considering the left-hand quotients induced by the ideals of  $p$-compact, weakly $p$-compact, unconditionally $p$-compact, approximable  or right $p$-nuclear operators with their respective weighted holomorphic ideals. On the other hand, the procedure is of interest when considering other operators ideals as it provides new weighted holomorphic ideals. This is the case of the ideal of Grothendieck weighted holomorphic mappings or the ideal of Rosenthal weighted holomorphic mappings, where the applicability of this construction is shown.
\end{abstract}

\maketitle

\section{Introduction and preliminaries}

Let $E$ and $F$ be complex Banach spaces, and let $U$ be an open subset of $%
E $. Let $\mathcal{H}(U,F)$ be the space of all holomorphic mappings from $U$
into $F$. A weight $v$ on $U$ is a strictly positive continuous function.
The space of weighted holomorphic mappings, denoted by $\mathcal{H}%
_{v}^{\infty }(U,F)$, is the Banach space consisting of all mappings $f\in 
\mathcal{H}(U,F)$ such that%
\begin{equation*}
\left\Vert f\right\Vert _{v}:=\sup \left\{ v(x)\left\Vert f(x)\right\Vert
\colon x\in U\right\} <\infty ,
\end{equation*}%
with the weighted supremum norm $\Vert \cdot \Vert _{v}$. In particular, $%
\mathcal{H}_{v}^{\infty }(U)$ is defined as$\ \mathcal{H}_{v}^{\infty }(U,%
\mathbb{C})$. Weighted holomorphic mappings on general (infinite dimensional) Banach spaces were introduced in \cite{GaMaRu} and \cite{Ru}, and since then many works have tried to describe their properties, differences and similarities with the finite dimensional case (see for instance \cite{BaDe1,BlaGaMi,  BaDe, BaDe3, JiRaVi,Jo} and the references therein). The lack of local compactness  forces the use of very different techniques for infinite dimensional domains that have been masterfully explored along the latest decades. 

On the other hand, Puhl \cite{Puhl1977} introduced left-hand operator ideals, that have been deeply used in the Theory of operator ideals along the years (e.g. \cite{CaDe,CaNa,JoLiO,Ki}).  Jim\'enez-Vargas and Ruiz-Casternado \cite{JimAysRuiz24} have successfully spread the notion of left-hand quotient of linear operator ideals to bounded holomorphic mappings, using it as a method for generating new bounded holomorphic ideals. 

Our aim is to combine both general settings --weighted holomorphic mappings and left-hand quotients-- to bring both theories into line introducing  a weighted holomorphic version of the concept of the
left-hand quotient of operator ideals. So, our starting point is centered in  weighted holomorphic ideals and their ``left-hand quotients'' with linear operator ideals. Several particular weighted holomorphic ideals have been considered in \cite{CaJiKe} and we will make good use of such examples for our purposes. Note that bounded  holomorphic mappings on $U$ can be seen as a particular case of weighted bounded holomorphic mappings just considering the weight $v(x)=1$ for all $x\in U$. 

Generalizing the case of bounded holomorphic mappings due to Mujica \cite{Mu}, the Banach space $(\mathcal{H}
_{v}^{\infty }(U), \|.\|_v)$ has a predual $\mathcal{G}_{v}^{\infty
}(U)\subset H_v^\infty (U)^*$ that satisfies the following linearization theorem  and properties (see \cite{BaDe3,Ru}): there exists a mapping $\Delta _{v}\in \mathcal{H}_{v}^{\infty
}(U,\mathcal{G}_{v}^{\infty }(U))$ such that for every $f\in \mathcal{H}%
_{v}^{\infty }(U,F)$, there exists a unique bounded linear  operator $T_{f}\in \mathcal{L}(%
\mathcal{G}_{v}^{\infty }(U),F)$ fulfilling $T_{f}\circ \Delta _{v}=f$.
Furthermore, $\left\Vert T_{f}\right\Vert =\left\Vert f\right\Vert _{v}$.
The mapping $\Delta _{v}$ is given by $\Delta _{v}(x)=\delta _{x}$, where  $\delta _{x}\colon \mathcal{H}_{v}^{\infty }(U)\rightarrow F$ is defined by $\delta _{x}(f)=f(x)$, for $f\in \mathcal{H}_{v}^{\infty }(U,F)$.

The linearization technique has been crucial in the last decades, as it has provided a method for using and even transferring known results from linear operators to the context of non linear mappings. Many authors have developed this linearization technique in different contexts, because of its power for handling nonlinear concepts through linear theory (see for instance \cite{BoPeRu,BeBuRu, BuBu, PeRuSa}). Here, we will use linearizations to describe left-hand quotients of weighted holomorphic mappings.

Let us recall
that a normed (Banach) ideal of weighted holomorphic mappings --in short, a normed (Banach) weighted holomorphic ideal-- is an assignment $\left[ \mathcal{I}^{\mathcal{H}_{v}^{\infty }},\Vert \cdot \Vert _{\mathcal{I}^{\mathcal{H}_{v}^{\infty }}}%
\right] $ which associates with every pair $(U,F)$, where $U$ is an open
subset of a complex Banach space $E$ and $F$ is a complex Banach space, a
set $\mathcal{I}^{\mathcal{H}_{v}^{\infty }}(U,F)\subseteq \mathcal{H}%
_{v}^{\infty }(U,F)$ and a function $\Vert \cdot \Vert _{\mathcal{I}^{%
\mathcal{H}_{v}^{\infty }}}\colon \mathcal{I}^{\mathcal{H}_{v}^{\infty
}}(U,F)\rightarrow \mathbb{R}_{0}^{+}$ satisfying the properties:

\begin{itemize}
\item[(P1)] $\left(\mathcal{I}^{\mathcal{H}_v^{\infty}}(U,F),\|\cdot\|_{%
\mathcal{I}^{\mathcal{H}_v^{\infty}}}\right)$ is a normed (Banach) space with $\|f\|_{%
\mathcal{I}^{\mathcal{H}_v^{\infty}}}\geq \|f\|_v$ for all $f\in\mathcal{I}^{%
\mathcal{H}_v^{\infty}}(U,F)$.

\item[(P2)] For any $h\in\mathcal{H}_v^{\infty}(U)$ and $y\in F$, the map $%
h\cdot y\colon x\mapsto h(x)y$ from $U$ to $F$ is in $\mathcal{I}^{\mathcal{H%
}_v^{\infty}}(U,F)$ with $\|h\cdot y\|_{\mathcal{I}^{\mathcal{H}%
_v^{\infty}}}=\|h\|_v\|y\|$.

\item[(P3)] The ideal property: if $V$ is an open subset of $E$ such that $%
V\subseteq U$, $h\in\mathcal{H}(V,U)$ with $c_v(h):=\sup_{x\in
V}(v(x)/v(h(x)))<\infty$ , $f\in \mathcal{I}^{\mathcal{H}_v^{\infty}}(U,F)$
and $T\in\mathcal{L}(F,G)$, where $G$ is a complex Banach space, then $%
T\circ f\circ h\in\mathcal{I}^{\mathcal{H}_v^{\infty}}(V,G)$ with $\|T\circ
f\circ h\|_{\mathcal{I}^{\mathcal{H}_v^{\infty}}}\leq \|T\|\|f\|_{\mathcal{I}%
^{\mathcal{H}_v^{\infty}}}c_v(h)$.
\end{itemize}

A normed (Banach) weighted holomorphic ideal $[\mathcal{I}^{\mathcal{H}%
_v^{\infty}},\|\cdot\|_{\mathcal{I}^{\mathcal{H}_v^{\infty}}}]$ is called:

\begin{itemize}
\item[(S)] Surjective if $f\in \mathcal{I}^{\mathcal{H}_{v}^{\infty }}(U,F)$
with $\left\Vert f\right\Vert _{\mathcal{I}^{\mathcal{H}_{v}^{\infty
}}}=\left\Vert f\circ \pi \right\Vert _{\mathcal{I}^{\mathcal{H}_{v}^{\infty
}}}$, whenever $f\in \mathcal{H}_{v}^{\infty }(U,F),\pi \in \mathcal{H}(V,U)$ is a surjective map, where $V$ is an open subset of a complex Banach space 
$G$ and $f\circ \pi \in \mathcal{I}^{\mathcal{H}_{v}^{\infty }}(V,F).$
\end{itemize}

In Section \ref{s1} left-hand quotients of weighted holomorphic mappings are defined and some basic properties are proved. We characterize when  left-hand quotients coincide with the whole space of weighted holomorphic mappings, and show that the most classical operators ideals as the ideal of all finite rank operators, the ideal of all approximable operators, the ideal of all compact operators, the ideal of all weakly compact operators (and their $p$-compact variants) or the ideal of right $p$-nuclear operators give rise to coincidence results whenever we consider their left-hand quotients with their respective weighted holomorphic ideals. We also show the relation of weighted holomorphic left-hand quotients with composition ideals.

Section \ref{sec} is devoted to give examples of weighted holomorphic left-hand quotients that provide new weighted holomorphic ideals. These examples are based in two well-known operator ideals: the ideal of all Grothendieck operators, denoted by $\mathcal O$, and the ideal of all  Rosenthal operators, denoted by $\mathcal R$. Concretely, we introduce the weighted holomorphic ideals formed by those weighted holomorphic mappings whose weighted holomorphic rank is (1) a Grothendieck set and (2) a Rosenthal set. Let us denote these weighted holomorphic ideals by    $\mathcal H_{v\mathcal O}^\infty $ and $\mathcal H_{v\mathcal R}^\infty $ respectively. With respect to $\mathcal H_{v\mathcal O}^\infty $ the main result shows that this ideal can be described as the left-hand quotient generated by the ideal of all weakly compact weighted holomorphic mappings and the ideal of all separable bounded linear operators, in other words, Grothendieck weighted holomorphic mappings are those whose composition with separable bounded linear operators gives rise to weakly compact weighted holomorphic mappings. Concerning $\mathcal H_{v\mathcal R}^\infty $, it is shown that this ideal can be described as the left-hand quotient generated by the ideal of all compact weighted holomorphic mappings  and the ideal of all completely continuous linear operators, in other words, Rosenthal weighted holomorphic mappings are those whose composition with completely continuous  linear operators gives rise to  compact weighted holomorphic mappings.

\section{Weighted holomorphic left-hand quotient ideals}\label{s1}

Let $\mathcal{I}$ and $\mathcal{J}$ be operator ideals. The left-hand quotient $\mathcal{I}
^{-1}\circ \mathcal{J}$ is  formed by all  bounded linear operators $%
T:E\rightarrow F$, whenever $E,$ $F$ are 
Banach spaces, such that  $A\circ T\in 
\mathcal{J}\left( E,G\right) $ for every Banach space $G$ and every $A\in \mathcal{I}\left( F,G\right)$. In such a case we write $T\in \mathcal{I}^{-1}\circ \mathcal{J}%
\left( E,F\right). $ 

Whenever $\left[ \mathcal{I},\left\Vert .\right\Vert _{\mathcal{I}}%
\right] $ and $\left[ \mathcal{J},\left\Vert .\right\Vert _{\mathcal{J}}%
\right] $ are normed operator ideals, the operator ideal $\mathcal{I}
^{-1}\circ \mathcal{J}$ is endowed with the norm
\begin{equation*}
\left\Vert T\right\Vert _{\mathcal{I}^{-1}\circ \mathcal{J}}=\sup \left\{
\left\Vert A\circ T\right\Vert _{\mathcal{J}}:A\in \mathcal{I}\left(
F,G\right) ,\left\Vert A\right\Vert _{\mathcal{I}}\leq 1\right\} ,
\end{equation*}%
where $G$ ranges over all Banach spaces and $A$ over all continuous linear operators in $ \mathcal{I}\left( F,G\right)$.

For general considerations on left-hand quotients we refer to 
 \cite{Pietsch1980}. 
 
 We first extend this notion to the weighted holomorphic setting.

\begin{definition}\label{Def1}
Let $\left[ \mathcal{A},\left\Vert .\right\Vert _{\mathcal{A}}\right] $ be a
normed operator ideal and let $\left[ \mathcal{I}^{\mathcal{H}_{v}^{\infty
}},\left\Vert .\right\Vert _{\mathcal{I}^{\mathcal{H}_{v}^{\infty }}}\right] 
$ be a normed weighted holomorphic ideal. A mapping $f\in \mathcal{H}%
_{v}^{\infty }(U,F)$ is said to belong to the weighted holomorphic left-hand
quotient $\mathcal{A}^{-1}\mathcal{\circ I}^{\mathcal{H}_{v}^{\infty }}$ and
will be written as $f\in \mathcal{A}^{-1}\mathcal{\circ I}^{\mathcal{H}%
_{v}^{\infty }}\left( U,F\right) $, if $A\circ f\in \mathcal{I}^{\mathcal{H}%
_{v}^{\infty }}\left( U,G\right) $ for all $A\in \mathcal{A}\left(
F,G\right) $ and all complex Banach space $G$. We define the weighted
holomorphic left-hand quotient norm as follows:%
\begin{equation*}
\left\Vert f\right\Vert _{\mathcal{A}^{-1}\mathcal{\circ I}^{\mathcal{H}%
_{v}^{\infty }}}=\sup \left\{ \left\Vert A\circ f\right\Vert _{\mathcal{I}^{%
\mathcal{H}_{v}^{\infty }}}:A\in \mathcal{A}\left( F,G\right) ,\left\Vert
A\right\Vert _{\mathcal{A}}\leq 1\right\} .
\end{equation*}
\end{definition}

Adapting  the proof of \cite[Proposition 2.1]{JimAysRuiz24}, it is not difficult to prove that 
the above supremum  is finite and, using standard arguments, it is not difficult to show that  $\left[ \mathcal{A}%
^{-1}\mathcal{\circ I}^{\mathcal{H}_{v}^{\infty }},\left\Vert .\right\Vert _{%
\mathcal{A}^{-1}\mathcal{\circ I}^{\mathcal{H}_{v}^{\infty }}}\right]$ is a normed weighted holomorphic ideal, which is complete whenever $\left[ \mathcal{I}^{\mathcal{H}_{v}^{\infty
}},\left\Vert .\right\Vert _{\mathcal{I}^{\mathcal{H}_{v}^{\infty }}}\right] 
$ is, and surjective whenever $\left[ \mathcal{I}^{\mathcal{H}_{v}^{\infty
}},\left\Vert .\right\Vert _{\mathcal{I}^{\mathcal{H}_{v}^{\infty }}}\right] 
$ is.

Let us start by showing some basic properties. 
As expected, an inclusion property can be established between weighted
holomorphic left-hand quotients through the inclusion of their corresponding
weighted holomorphic ideals. For two normed weighted holomorphic  ideals $\left[ \mathcal{I}_{1}^{\mathcal{H}_{v}^{\infty }},\left\Vert
.\right\Vert _{\mathcal{I}_{1}^{\mathcal{H}_{v}^{\infty }}}\right] $ and $%
\left[ \mathcal{I}_{2}^{\mathcal{H}_{v}^{\infty }},\left\Vert .\right\Vert _{%
\mathcal{I}_{2}^{\mathcal{H}_{v}^{\infty }}}\right] $, we
write $\left[ \mathcal{I}_{1}^{\mathcal{H}_{v}^{\infty }},\left\Vert .\right\Vert _{%
\mathcal{I}_{1}^{\mathcal{H}_{v}^{\infty }}}\right] \leq \left[ \mathcal{I}%
_{2}^{\mathcal{H}_{v}^{\infty }},\left\Vert .\right\Vert _{\mathcal{I}_{2}^{%
\mathcal{H}_{v}^{\infty }}}\right] $ if $\mathcal{I}_{1}^{\mathcal{H}_{v}^{\infty }}\subseteq \mathcal{I}_{2}^{\mathcal{H}_{v}^{\infty }}$ and $\left\Vert
f\right\Vert _{\mathcal{I}_{2}^{\mathcal{H}_{v}^{\infty }}}\leq \left\Vert f\right\Vert _{\mathcal{I}_{1}^{\mathcal{H}_{v}^{\infty }}}$ for all $f\in \mathcal{I}_{1}^{\mathcal{H}_{v}^{\infty }}.$

\begin{proposition}\label{inclusion}
Let $\left[ \mathcal{I}_{1}^{\mathcal{H}_{v}^{\infty }},\left\Vert
.\right\Vert _{\mathcal{I}_{1}^{\mathcal{H}_{v}^{\infty }}}\right] $ and $%
\left[ \mathcal{I}_{2}^{\mathcal{H}_{v}^{\infty }},\left\Vert .\right\Vert _{%
\mathcal{I}_{2}^{\mathcal{H}_{v}^{\infty }}}\right] $ be normed weighted
holomorphic\ ideals such that%
\begin{equation*}
\left[ \mathcal{I}_{1}^{\mathcal{H}_{v}^{\infty }},\left\Vert .\right\Vert _{%
\mathcal{I}_{1}^{\mathcal{H}_{v}^{\infty }}}\right] \leq \left[ \mathcal{I}%
_{2}^{\mathcal{H}_{v}^{\infty }},\left\Vert .\right\Vert _{\mathcal{I}_{2}^{%
\mathcal{H}_{v}^{\infty }}}\right] ,
\end{equation*}%
then for any normed operator ideal $\left[ \mathcal{A},\left\Vert
.\right\Vert _{\mathcal{A}}\right] $, the following inclusion holds:%
\begin{equation*}
\left[ \mathcal{A}^{-1}\mathcal{\circ I}_{1}^{\mathcal{H}_{v}^{\infty
}},\left\Vert .\right\Vert _{\mathcal{A}^{-1}\mathcal{\circ I}_{1}^{\mathcal{%
H}_{v}^{\infty }}}\right] \leq \left[ \mathcal{A}^{-1}\mathcal{\circ I}_{2}^{%
\mathcal{H}_{v}^{\infty }},\left\Vert .\right\Vert _{\mathcal{A}^{-1}%
\mathcal{\circ I}_{2}^{\mathcal{H}_{v}^{\infty }}}\right] .
\end{equation*}
\end{proposition}

Having in mind  that $\left[ \mathcal{H}_{v}^{\infty },\left\Vert \cdot\right\Vert _{v}\right] $ is trivially a Banach weighted holomorphic ideal, by the previous result
\begin{equation*}
\left[ \mathcal{A}^{-1}\mathcal{\circ I}^{\mathcal{H}_{v}^{\infty
}},\left\Vert .\right\Vert _{\mathcal{A}^{-1}\mathcal{\circ I}^{\mathcal{H}%
_{v}^{\infty }}}\right] \leq \left[ \mathcal{A}^{-1}\mathcal{\circ H}%
_{v}^{\infty },\left\Vert .\right\Vert _{\mathcal{A}^{-1}\mathcal{\circ H}%
_{v}^{\infty }}\right]
\end{equation*}%
for any normed operator ideal $\left[ \mathcal{A},\left\Vert .\right\Vert _{%
\mathcal{A}}\right] $ and any normed weighted holomorphic ideal $\left[ \mathcal{I}^{\mathcal{H}_{v}^{\infty }},\left\Vert
.\right\Vert _{\mathcal{I}^{\mathcal{H}_{v}^{\infty }}}\right] $. Hence, $\mathcal{A}^{-1}\mathcal{\circ H}_{v}^{\infty }\ $ can be considered the
largest weighted holomorphic left-hand quotient for any normed operator
ideal $\mathcal{A}$.

Let $\left[ \mathcal{A},\left\Vert .\right\Vert _{\mathcal{A}}\right] $ be a
normed operator ideal and let $\left[ \mathcal{I}^{\mathcal{H}_{v}^{\infty
}},\left\Vert .\right\Vert _{\mathcal{I}^{\mathcal{H}_{v}^{\infty }}}\right] 
$ be a normed weighted holomorphic ideal. We say that $\mathcal{I}^{\mathcal{%
H}_{v}^{\infty }}$ has the linearization property (\textit{LP}, for short)
in $\mathcal{A}$ if given $f\in \mathcal{H}_{v}^{\infty }\left( U,F\right) $%
, we have that $f\in \mathcal{I}^{\mathcal{H}_{v}^{\infty }}\left(
U,F\right) $ if and only if $T_{f}\in \mathcal{A}(\mathcal{G}_{v}^{\infty
}(U),F)$, in whose case $\left\Vert f\right\Vert _{\mathcal{I}^{\mathcal{H}%
_{v}^{\infty }}}=\left\Vert T_f\right\Vert _{\mathcal{A}}.$

The following result shows how strong  the assumption of having the LP is: if
the weighted holomorphic ideal $\mathcal{I}^{\mathcal{H}_{v}^{\infty }}$ has the LP in $\mathcal A$ then, the left-hand quotient $\mathcal A^{-1}\circ \mathcal{I}^{\mathcal{H}_{v}^{\infty }}$ coincides with the whole ideal $\mathcal{H}_{v}^{\infty }$.

\begin{proposition}\label{coincidence}
\label{3}Let $\left[ \mathcal{A},\left\Vert .\right\Vert _{\mathcal{A}}%
\right] $ be a Banach operator ideal and let $\left[ \mathcal{I}^{\mathcal{H}%
_{v}^{\infty }},\left\Vert .\right\Vert _{\mathcal{I}^{\mathcal{H}%
_{v}^{\infty }}}\right] $ be a normed weighted holomorphic ideal. Then,%
\begin{equation*}
\left[ \mathcal{A}^{-1}\mathcal{\circ I}^{\mathcal{H}_{v}^{\infty
}},\left\Vert .\right\Vert _{\mathcal{A}^{-1}\mathcal{\circ I}^{\mathcal{H}%
_{v}^{\infty }}}\right] \leq \left[ \mathcal{H}_{v}^{\infty },\left\Vert
.\right\Vert _{v}\right] .
\end{equation*}

Furthermore%
\begin{equation*}
\left[ \mathcal{A}^{-1}\mathcal{\circ I}^{\mathcal{H}_{v}^{\infty
}},\left\Vert .\right\Vert _{\mathcal{A}^{-1}\mathcal{\circ I}^{\mathcal{H}%
_{v}^{\infty }}}\right] =\left[ \mathcal{H}_{v}^{\infty },\left\Vert
.\right\Vert _{v}\right]
\end{equation*}%
if $\mathcal{I}^{\mathcal{H}_{v}^{\infty }}$ has the LP in $\mathcal{A}$.
\end{proposition}

\begin{proof}
Let $f\in \mathcal{A}^{-1}\mathcal{\circ I}^{\mathcal{H}_{v}^{\infty
}}\left( U,F\right) $, we have that $\left\Vert f\right\Vert _{v}\leq
\left\Vert f\right\Vert _{\mathcal{A}^{-1}\mathcal{\circ I}^{\mathcal{H}%
_{v}^{\infty }}}$. Indeed, notice that $f\in \mathcal{H}_{v}^{\infty }\left(
U,F\right) $ and $A\circ f\in \mathcal{I}^{\mathcal{H}_{v}^{\infty }}\left(
U,G\right) $ for all $A\in \mathcal{A}\left( F,G\right) $ and all 
complex Banach space $G$. For any $x\in U$ consider a functional $y^{\ast }_x$ in the closed unit ball of the dual space $F^{\ast }$ of $F$ so that $\left\Vert
f\left( x\right) \right\Vert =\left\vert y^{\ast }_x\left( f\left( x\right)
\right) \right\vert $. Taking into account that $\left[ \mathcal{A}%
,\left\Vert .\right\Vert _{\mathcal{A}}\right] $ is an operator ideal,
we have that  $y^{\ast }_x$ belongs to $\mathcal{A}\left( F,\mathbb{C}\right) $
with $\left\Vert y^{\ast }_x\right\Vert _{\mathcal{A}}=\left\Vert
y^{\ast }_x\right\Vert \leq 1$. Then, $y^*_x\circ f\in \mathcal{ I}^{\mathcal{H}_{v}^{\infty
}}$ and $\|y_x^{\ast}\circ f\|_{\mathcal I^{\mathcal H_v^\infty}}\leq \|f\|_{\mathcal A^{-1}\circ \mathcal I^{\mathcal H_v^\infty}}$. Hence, 

$$
v(x)\|f(x)\|=v(x)|y_x^*(f(x))|\leq \|y_x^*\circ f\|_v\leq \|y_x^*\circ f\|_{\mathcal I^{\mathcal H_v^\infty}}\leq \|f\|_{\mathcal A^{-1}\circ \mathcal I^{\mathcal H_v^\infty}}
$$
and taking the supremum over all $x\in U$, we conclude that $\left\Vert
f\right\Vert _{v}\leq \left\Vert f\right\Vert _{\mathcal{A}^{-1}\mathcal{%
\circ I}^{\mathcal{H}_{v}^{\infty }}}.$

Let us now assume that $\mathcal{I}^{\mathcal{H}_{v}^{\infty }}$ has the 
\textit{LP} in the operator ideal $\mathcal{A}$. Let $f\in \mathcal{H}%
_{v}^{\infty }\left( U,F\right) $ and let $A\in \mathcal{A}\left( F,G\right) 
$, for a complex Banach space $G$. It is not difficult to see that $A\circ
f\in \mathcal{H}_{v}^{\infty }\left( U,G\right) $. Consider the linearizations $T_f\in \mathcal{L}(\mathcal{G}_{v}^{\infty }(U),F)$ and $T_{A\circ
f}\in \mathcal{L}(\mathcal{G}_{v}^{\infty }(U),G)$. We know that  $\left\Vert
T_{A\circ f}\right\Vert =\left\Vert A\circ f\right\Vert _{v}$ and $\left\Vert T_{f}\right\Vert =\left\Vert f\right\Vert _{v}$. Since
\begin{equation*}
T_{A\circ f}\circ \Delta _{v}=A\circ f=A\circ T_{f}\circ \Delta _{v},
\end{equation*}
by the uniqueness of the linearization, $T_{A\circ f}=A\circ T_{f}$. Thus, by the ideal property of $\mathcal{A}$,
we have $T_{A\circ f}\in \mathcal{A}(\mathcal{G}_{v}^{\infty }(U),G)$ with%
\begin{equation*}
\left\Vert T_{A\circ f}\right\Vert _{\mathcal{A}}=\left\Vert A\circ
T_{f}\right\Vert _{\mathcal{A}}\leq \left\Vert A\right\Vert _{\mathcal{A}%
}\left\Vert T_{f}\right\Vert \leq \left\Vert A\right\Vert _{\mathcal{A}%
}\left\Vert f\right\Vert _{v}.
\end{equation*}%
Since $\mathcal{I}^{\mathcal{H}_{v}^{\infty }}$ has the \textit{LP} in $%
\mathcal{A}$  we have $A\circ f\in \mathcal{I}^{\mathcal{H}_{v}^{\infty }}\left(
U,G\right) $ with $\left\Vert A\circ f\right\Vert _{\mathcal{I}^{\mathcal{H}%
_{v}^{\infty }}}=\left\Vert T_{A\circ f}\right\Vert _{\mathcal{A}}$. Therefore,  $f\in 
\mathcal{A}^{-1}\mathcal{\circ I}^{\mathcal{H}_{v}^{\infty }}\left(
U,F\right) $ and 
\begin{eqnarray*}
\left\Vert f\right\Vert _{\mathcal{A}^{-1}\mathcal{\circ I%
}^{\mathcal{H}_{v}^{\infty }}}&=&\sup\{ \|A\circ f\|_{\mathcal I^{\mathcal H_v^\infty}} : A\in \mathcal A, \|A\|_{\mathcal A}\leq 1\}\\
&=&  \sup\{ \|T_{A\circ f}\|_{\mathcal A} : A\in \mathcal A, \|A\|_{\mathcal A}\leq 1\}\\
&=&\sup\{ \|A\circ T_f\|_{\mathcal A} : A\in \mathcal A, \|A\|_{\mathcal A}\leq 1\}\leq \|T_f\|= \left\Vert f\right\Vert _{v}
 \end{eqnarray*}
  as desired.
\end{proof}

Let us first see several cases where the above result applies and gives coincidence results for weighted holomorphic ideals.   A set $K\subset F$ is:

\medskip
$\bullet$  relatively $p$-compact, $p\in (1,\infty)$,  if there is a $p$-summable sequence $(x_n)_n$ in $F$   such that $K\subset p-conv(x_n):=\{ \sum_{n=1}^\infty a_nx_n: \|(a_n)_n\|_{p^*}=(\sum_{n=1}^\infty |a_n|^{p^*})^{1/p^*}\leq 1\}$.

\medskip
$\bullet$  relatively $1$-compact if there is a $1$-summable sequence $(x_n)_n$ in $F$   such that $K\subset 1-conv(x_n):=\{ \sum_{n=1}^\infty a_nx_n: \|(a_n)_n\|_{\infty}=\sup_n |a_n|\leq 1\}$.

\medskip
$\bullet$  relatively $\infty$-compact if there is a  a null sequence $(x_n)_n$ in $F$  such that $K\subset \infty-conv(x_n):=\{ \sum_{n=1}^\infty a_nx_n: \|(a_n)_n\|_{1}=\sum_{n=1}^\infty |a_n|\leq 1\}$.

\medskip
By a well known theorem due to Grothendieck, relatively $\infty$-compact sets reduce to all compact sets.  Besides,  if $1\leq p\leq q\leq \infty$ then, any $p$-compact set is $q$-compact.

Similarly one can define relatively unconditionally (weakly) $p$-compact sets   just considering unconditionally (resp. weakly) $p$-summable sequences $(x_n)_n$ for $p\in [1,\infty)$ or   null (resp. weakly null) sequences $(x_n)_n$ for $p=\infty$.

Let $\mathcal K_p, \mathcal K_{wp}$ and $\mathcal K_{up}$ be the ideals of all $p$-compact (maps the unit ball into a relatively $p$-compact set), weakly $p$-compact (maps the unit ball into a relatively weakly $p$-compact set) and unconditionally $p$-compact (maps the unit ball into a relatively unconditionally $p$-compact set) linear operators respectively. For $i=p,wp,up$, and $p\in [1,\infty]$, $\mathcal K_i$ is a Banach ideal when endowed with the norm
$$
\|T\|_{\mathcal K_i}:=\inf \| (x_n)_n\|_{i}, \quad T\in \mathcal K_i(E;F)
$$
where the infimum is taken over all sequences $(x_n)_n$ as above such that $T(B_E)\subset p-conv(x_n)$.

It is clear that $\mathcal K_\infty$ coincides with the ideal $\mathcal K$ of all compact operators   and, similarly, $\mathcal K_{w\infty}=\mathcal K_w$ is the ideal of all wealy compact operators, and $\mathcal K_{u\infty}=\mathcal K_u$ is the ideal of all unconditional compact operators.

\begin{example}   In \cite{CaJiKe} the spaces $\mathcal H_{v \mathcal K_p}^\infty(U,F),\mathcal H_{v \mathcal K_{wp}}^\infty(U,F)$ and $\mathcal H_{v \mathcal K_{up}}^\infty(U,F)$ of all weighted holomorphic mappings $f:U\to F$ for which $(vf)(U)$ is relatively $p$-compact, relatively weakly $p$-compact and relatively unconditionally $p$-compact respectively form Banach ideals of weighted holomorphic mappings with the respective norms 
$$
\|f\|_{\mathcal H_{v\mathcal K_i}^\infty}=\inf\{ \|(x_n)_n\|_i\} \quad i=p,wp,up
$$
where the infimum is taken over all $(x_n)_n$ such that $(vf)(U)\subset p-conv(x_n)$ and being $(x_n)_n$ $p$-summable if $i=p$, unconditionally $p$-summable if $i=up$ and weakly $p$-summable if $i=wp$.

 It is easy to realise that $[\mathcal I^{-1}\circ \mathcal H_{v\mathcal I}^\infty, \|\cdot\|_{\mathcal I^{-1}\circ \mathcal H_{v\mathcal I}^\infty}]=[\mathcal H_v^\infty,\|\cdot\|_\infty]$ for $\mathcal I=\mathcal K_p, \mathcal K_{wp}, \mathcal K_{up}$.

Proposition \ref{coincidence} fits with  the above coincidences as  all these ideals have the LP in the operator ideal of $p$-compact, weakly $p$-compact and unconditionally $p$-compact operators respectively. That is, if $f\in \mathcal H_v^\infty(U,F)$ and we write $\mathcal I=\mathcal K_p, \mathcal K_{wp}, \mathcal K_{up}$, $p\in [1,\infty]$, then $f\in \mathcal H_{v\mathcal I}^\infty (u,F)$ if, and only if, its linearization $T_f\in \mathcal I(\mathcal G_v^\infty(u),F)$. Besides, $\| f\|_{\mathcal H_{v\mathcal I}^\infty}=\|T_f\|_{\mathcal I}$ (see \cite[Theorem 2.9]{CaJiKe}).

 \end{example}

\begin{example}

Consider the  operator ideal $\mathcal F$ of all finite rank operators. Even not being a closed ideal,  it is clear that the linear hull of $v T\circ f (U)$ is a finite dimensional subspace of $G$ whenever $f\in \mathcal H_v^\infty (U,F)$ and $T\in \mathcal F(F,G)$. If we set $\mathcal H_{v\mathcal F}^\infty(U,F)$ formed by all $f\in \mathcal H_{v}^\infty(U,F)$ such that the linear hull of $(v f) (U)$ is a finite dimensional subspace of $F$, then we can write $[\mathcal F^{-1}\circ \mathcal H_{v\mathcal F}^\infty, \|\cdot\|_{\mathcal F^{-1}\circ \mathcal H_{v\mathcal F}^\infty}]=[\mathcal H_v^\infty,\|\cdot\|_\infty]$.

Consider now the closed operator ideal $\overline{\mathcal F}$  formed by all approximable linear operators. A mapping $f\in \mathcal H_v^\infty (U,F)$ is called approximable if there exists a sequence $(f_n)_n\in \mathcal H_{v\mathcal F}^\infty (U,F)$ such that $\|f_n-f\|_v\to 0$ as $n\to\infty$. Approximable weighted holomorphic mappings were introduced in \cite{CaJiKe}. In \cite[Proposition 2.24]{CaJiKe} it is proved that the Banach weighted holomorphic ideal $\mathcal H_{v\overline{\mathcal F}}^\infty$ formed by all approximable bounded weighted holomorphic mappings with the weighted norm, has the LP in $\overline{\mathcal F}$. So, an application of Proposition \ref{coincidence} (or Proposition \ref{inclusion} and the above) concludes that $[\overline{\mathcal F}^{-1}\circ \mathcal H_{v\overline{\mathcal F}}^\infty, \|\cdot\|_{\mathcal F^{-1}\circ \mathcal H_{v\mathcal F}^\infty}]=[\mathcal H_v^\infty,\|\cdot\|_\infty]$.

\end{example}

\begin{example}
Other examples where Proposition \ref{coincidence} applies to reinforce coincidence results are given when considering the ideal of all right $p$-nuclear  weighted holomorphic mappings as introduced in \cite{CaJiKe}. Let $1<p,p^*<\infty $ with $\frac 1p+\frac 1{p^*}=1$. A mapping $f\in \mathcal H_{v}^\infty(U,F)$ is right $p$-nuclear if there exists a weakly $p^*$-summable  sequence $(g_n)_n$ in $\mathcal H_v^\infty(U)$ and a $p$-summable sequence $(y_n)_n$ in $F$  such that $f=\sum_{n=1}^\infty g_n\cdot y_n$ in $(\mathcal H_v^\infty (U,F),\|\cdot\|_v)$. We denote $\mathcal H_{v\mathcal N_p}^\infty(U,F)$ the space of all right $p$-nuclear  weighted holomorphic mappings $f:U\to F$ endowed with the norm
$$
\|f\|_{\mathcal H_{v\mathcal N_p}^\infty}:=\inf\{\|(g_n)_n\|_{p^*}^w\|(y_n)_n\|_p\},
$$
where the infimum is taken over all right $p$-nuclear  weighted holomorphic representations of $f$. This definition derives from the operator ideal $\mathcal N^p$ of all right $p$-nuclear linear operators, just considering sequences  $(g_n)_n$ of holomorphic functions instead of continuous linear functionals in $E^*$.

By \cite[Proposition 2.28]{CaJiKe} the ideal $\mathcal H_{v\mathcal N_p}^\infty$ has the LP in $\mathcal N^p$ . Hence, by Proposition \ref{coincidence} 
$$
[(\mathcal N^p)^{-1}\circ \mathcal H_{v\mathcal N_p}^\infty, \|\cdot \|_{(\mathcal N^p))^{-1}\circ \mathcal H_{v\mathcal N_p}^\infty}]=[\mathcal H_v^\infty,\|\cdot\|_v].
$$
\end{example}

Let us summarize the above situations by proving what is happening behind all these examples, as show the next two resuls. We first recall the composition method for generating weighted holomorphic
ideals. Given an operator ideal $\mathcal{A}$ and a  weighted holomorphic ideal $\mathcal I^{\mathcal H_v^\infty}$, a map $f\in \mathcal{H}%
_{v}^{\infty }(U,F)$ belongs to $ \mathcal{A}\circ \mathcal I^{\mathcal H_v^\infty}(U,F)$ if there exist a complex Banach space $G$, an operator 
$A\in \mathcal{A}\left( G,F\right) $, and a map $g\in \mathcal I^{\mathcal H_v^\infty}(U,G)$, such that $f=A\circ g$. Let   $\mathcal{A}\circ 
\mathcal I^{\mathcal H_v^\infty}$ denote the composition ideal formed by all such $f$.

If $\left[ \mathcal{A},\left\Vert .\right\Vert _{\mathcal{A}}\right] $ is a
normed operator ideal, $[\mathcal I^{\mathcal H_v^\infty}, \|\cdot\|_{\mathcal I^{\mathcal H_v^\infty}}]$ is a normed weighted holomorphic ideal and $f\in \mathcal{A}\circ \mathcal I^{\mathcal H_v^\infty}$,
we define its norm as%
\begin{equation*}
\left\Vert f\right\Vert _{\mathcal{A}\circ \mathcal I^{\mathcal H_v^\infty}}=\inf
\left\{ \left\Vert A\right\Vert _{\mathcal{A}}\left\Vert g\right\Vert
_{\mathcal I^{\mathcal H_v^\infty}}\right\} \text{,}
\end{equation*}%
where the infimum is taken over all factorizations of $f$ as described above.

\begin{theorem}\label{caract}
Let $\left[ \mathcal{A},\left\Vert .\right\Vert _{\mathcal{A}}%
\right] $ be a normed operator ideal and let $\left[ \mathcal{I}^{\mathcal{H}%
_{v}^{\infty }},\left\Vert .\right\Vert _{\mathcal{I}^{\mathcal{H}%
_{v}^{\infty }}}\right] $ be a normed weighted holomorphic ideal. Then, 
$$
\left[ \mathcal{A}\circ\mathcal{ I}^{\mathcal{H}_{v}^{\infty
}},\left\Vert .\right\Vert _{\mathcal{A}\mathcal{\circ I}^{\mathcal{H}%
_{v}^{\infty }}}\right] \leq \left[\mathcal I^{ \mathcal{H}_{v}^{\infty }},\left\Vert
.\right\Vert _{\mathcal I^{\mathcal{H}_{v}^{\infty }}}\right] 
\mbox{ if, and only if, }\left[\mathcal A^{-1}\circ \mathcal I^{\mathcal H_v^\infty}, \|\cdot \|_{\mathcal A^{-1}\circ \mathcal I^{\mathcal H_v^\infty}}\right]=\left[\mathcal H_v^\infty,\|\cdot\|_v\right].$$
\end{theorem}

\begin{proof}
First assume that $ \left[ \mathcal{A}\circ\mathcal{ I}^{\mathcal{H}_{v}^{\infty
}},\left\Vert .\right\Vert _{\mathcal{A}\mathcal{\circ I}^{\mathcal{H}%
_{v}^{\infty }}}\right] \leq \left[\mathcal I^{ \mathcal{H}_{v}^{\infty }},\left\Vert
.\right\Vert _{\mathcal I^{\mathcal{H}_{v}^{\infty }}}\right]$. Let $f\in H_v^\infty (U,F)$ and take $A\in \mathcal A(F,G)$. Then, 
$$
A\circ f\in \mathcal A\circ \mathcal H_v^\infty(U,G)\subset \mathcal I^{ \mathcal{H}_{v}^{\infty }}(U,F).
$$
This means that $f\in \mathcal A^{-1}\circ \mathcal I^{\mathcal H_v^\infty}(U,F)$. Moreover, if $\|A\|_{\mathcal A}\leq 1$ then from the hypothesis of being $\left\Vert
.\right\Vert _{\mathcal I^{\mathcal{H}_{v}^{\infty }}}\leq \left\Vert .\right\Vert _{\mathcal{A}\mathcal{\circ I}^{\mathcal{H}%
_{v}^{\infty }}}$ we get
$$
\|A\circ f\|_{\mathcal I^{\mathcal H_v^\infty}}\leq \| A\circ f\|_{\mathcal A\circ \mathcal H_v^\infty}\leq \| A\|_{\mathcal A}\| f\|_v\leq \| f\|_v.
$$
Then, taking supremum on $A$ we obtain
$$
\|f\|_{\mathcal A^{-1}\circ \mathcal I^{\mathcal H_v^\infty}}\leq \|f\|_v.
$$
Hence, the equality of the norms $\|f\|_{\mathcal A^{-1}\circ \mathcal I^{\mathcal H_v^\infty}}= \|f\|_v$ follows.

Reciprocally, any $f\in \mathcal A\circ \mathcal H_v^\infty(U,G)$ can be written as $f=A\circ g$, with $A\in \mathcal A(F,G)$ and $g\in \mathcal H_v^\infty (U,F)$ for some suitable Banach space $F$ that depends on $f$. Then, by hypothesis $g\in  \mathcal H_v^\infty (U,F)=\mathcal A^{-1}\circ \mathcal I^{\mathcal H_v^\infty}(U,F)$ and $\|g\|_v=\|g\|_{\mathcal A^{-1}\circ \mathcal I^{\mathcal H_v^\infty}}$. Hence, $f=A\circ g\in \mathcal{ I}^{\mathcal{H}_{v}^{\infty
}}(U,G)$ and 
$$
\|f\|_{\mathcal I^{\mathcal H_v^\infty}}=\|A\circ g\|_{\mathcal I^{\mathcal H_v^\infty}}\leq \|A\|_{\mathcal A}\|g\|_{\mathcal A^{-1}\circ \mathcal I^{\mathcal H_v^\infty}}=\|A\|_{\mathcal A}\|g\|_v.
$$
Taking the infimum on $A$ we get $\|f\|_{\mathcal I^{\mathcal H_v^\infty}}\leq \| f\|_{\mathcal A\circ \mathcal I^{\mathcal H_v^\infty}}$ as desired.
\end{proof}

\begin{proposition}\label{LinPro}
Let $\left[ \mathcal{A},\left\Vert .\right\Vert _{\mathcal{A}}%
\right] $ be a normed operator ideal and let $\left[ \mathcal{I}^{\mathcal{H}%
_{v}^{\infty }},\left\Vert .\right\Vert _{\mathcal{I}^{\mathcal{H}%
_{v}^{\infty }}}\right] $ be a normed weighted holomorphic ideal. Then, $ \mathcal{I}^{\mathcal{H}%
_{v}^{\infty }}$ has the LP in $\mathcal{A}$ if, and only if, the equality  $\left[ \mathcal{I}^{\mathcal{H}%
_{v}^{\infty }},\left\Vert .\right\Vert _{\mathcal{I}^{\mathcal{H}%
_{v}^{\infty }}}\right] =\left[\mathcal{A}\circ \mathcal H_v^\infty, \|\cdot\|_{\mathcal{A}\circ \mathcal H_v^\infty}\right]$ holds.
\end{proposition}

\begin{proof} Assume first that $ \mathcal{I}^{\mathcal{H}%
_{v}^{\infty }}$ has the LP in $\mathcal{A}$.

Let $f\in \mathcal{I}^{\mathcal{H}%
_{v}^{\infty }}(U,F)$ and consider its decomposition $f=T_f\circ \Delta_v$. If we assume that  $ \mathcal{I}^{\mathcal{H}%
_{v}^{\infty }}$ has the LP in $\mathcal{A}$ then $T_f\in \mathcal A(\mathcal G_v^\infty(U),F)$. Hence, $f\in \mathcal A\circ \mathcal H_v^\infty(U,F)$. Besides,
$$
\|f\|_{\mathcal A \circ \mathcal H_v^\infty}\leq \|T_f\|_{\mathcal A}\|\Delta_v\|_v\leq \|f\|_{\mathcal I^{\mathcal H_v^\infty}}.
$$

Reciprocally, If $f \in \mathcal A\circ \mathcal H_v^\infty (U,F)$ then $f=A\circ g$ for some $A\in \mathcal A(G,F)$, some $g\in \mathcal H_v^\infty(U,G)$ and some Banach space $G$. By the ideal property, $A\circ T_g\in \mathcal A(\mathcal G_v^\infty(U),F)$ and by the uniqueness of the linearization $A\circ T_g=T_f$. Hence, $T_f\in \mathcal A(\mathcal G_v^\infty(U),F)$, and by the linearization property $f\in  \mathcal I^{\mathcal H_v^\infty}(U,F)$. Moreover, by the ideal property 
$$
\|f\|_{\mathcal I^{\mathcal H_v^\infty}}=\|T_f\|_{\mathcal A}=\|A\circ T_g\|_{\mathcal A}\leq \|A\|_{\mathcal A}\|T_g\|= \|A\|_{\mathcal A}\|g\|_v,
$$
and taking infimum on $A$ we get that $\|f\|_{\mathcal I^{\mathcal H_v^\infty}}\leq \|f\|_{\mathcal A\circ \mathcal H_v^\infty }$.

If we now assume that $\left[ \mathcal{I}^{\mathcal{H}%
_{v}^{\infty }},\left\Vert .\right\Vert _{\mathcal{I}^{\mathcal{H}%
_{v}^{\infty }}}\right] =\left[\mathcal{A}\circ \mathcal H_v^\infty, \|\cdot\|_{\mathcal{A}\circ \mathcal H_v^\infty}\right]$, then \cite[Theorem 2.7]{CaJiKe} gives the result.
\end{proof}

Using Propositions \ref{coincidence} and \ref{LinPro} it now easily follows the following result that shows the relation of the   left-hand quotients of weighted holomorphic mappings and the composition ideals.

\begin{theorem}
Let $\left[ \mathcal{A},\left\Vert \cdot\right\Vert _{\mathcal{A}}\right] $ be a
normed operator ideal and let $\left[ \mathcal{I}^{\mathcal{H}_{v}^{\infty
}},\left\Vert \cdot\right\Vert _{\mathcal{I}^{\mathcal{H}_{v}^{\infty }}}\right] 
$ be a normed  weighted holomorphic ideal that has the LP in $\mathcal A$. Then,
 $[\mathcal A\circ(\mathcal A^{-1}\circ\mathcal I^{\mathcal H_v^\infty}), \|\cdot\|_{ \mathcal A\circ(\mathcal A^{-1}\circ\mathcal I^{\mathcal H_v^\infty})}]= [\mathcal I^{\mathcal H_v^\infty}, \|\cdot\|_{\mathcal I^{\mathcal H_v^\infty}}]$.

\end{theorem}

\begin{lemma}\label{asoc}
Let $\left[ \mathcal{A},\left\Vert .\right\Vert _{\mathcal{A}}%
\right] $ and $\left[ \mathcal{I},\left\Vert .\right\Vert _{\mathcal{I}}%
\right] $ be  normed operator ideals. Then, $[(\mathcal A^{-1}\circ \mathcal I)\circ \mathcal H_v^\infty, \| \cdot\|_{(\mathcal A^{-1}\circ \mathcal I)\circ \mathcal H_v^\infty}]=[\mathcal A^{-1}\circ (\mathcal I\circ \mathcal H_v^\infty), \| \cdot\|_{\mathcal A^{-1}\circ (\mathcal I\circ \mathcal H_v^\infty)}]$.
\end{lemma}

\begin{proof} Let $f\in \mathcal H_v^\infty(U,F)$.
By definition, $f$ belongs to $\mathcal A^{-1}\circ (\mathcal I\circ \mathcal H_v^\infty)(U,F)$ if  $B\circ f\in \mathcal I\circ \mathcal H_v^\infty(U;G)$ for any $B\in \mathcal A(F,G)$ and any Banach space $G$. By \cite[Theorem 2.7]{CaJiKe} this is equivalent to $T_{B\circ f}$ belonging to $\mathcal I(\mathcal G_v^\infty(U),G)$ for any $B\in \mathcal A(F,G)$ and any Banach space $G$. Since $T_{B\circ f}=B\circ T_f$, the above means that $T_f \in \mathcal A^{-1}\circ \mathcal I(\mathcal G_v^\infty(U),F)$ and, again by \cite[Theorem 2.7]{CaJiKe}, this is equivalent to saying that $f$ belongs to $(\mathcal A^{-1}\circ \mathcal I)\circ \mathcal H_v^\infty(U,F)$.

In this case, 
\begin{eqnarray*}
\|f\|_{(\mathcal A^{-1}\circ \mathcal I)\circ \mathcal H_v^\infty}&=& \|T_f\|_{\mathcal A^{-1}\circ \mathcal I}=\sup\{ \|A\circ T_f\|_{\mathcal I}: A\in \mathcal A, \|A\|_{\mathcal A}\leq 1\}\\
&=&\sup\{ \|T_{A\circ f}\|_{\mathcal I}: A\in \mathcal A, \|A\|_{\mathcal A}\leq 1\}\\
&=&\sup\{ \|A\circ f\|_{\mathcal I\circ \mathcal H_v^\infty}: A\in \mathcal A, \|A\|_{\mathcal A}\leq 1\}\\
&=& \|f\|_{\mathcal A^{-1}\circ (\mathcal I\circ \mathcal H_v^\infty)}\\
\end{eqnarray*}
\end{proof}

As a consequence of Proposition \ref{LinPro}  and  Lemma \ref{asoc} left-hand quotients of an operator ideal and a weighted holomorphic ideal
with the LP in a certain operator ideal can be interpreted as a composition
ideal. 

\begin{corollary}\label{com}
Let $\left[ \mathcal{A},\left\Vert .\right\Vert _{\mathcal{A}}\right] $ and $\left[ \mathcal{I},\left\Vert .\right\Vert _{%
\mathcal{I}}\right] $ be normed operator ideals, and let $\left[ \mathcal{I}^{%
\mathcal{H}_{v}^{\infty }},\left\Vert .\right\Vert _{\mathcal{I}^{\mathcal{H}%
_{v}^{\infty }}}\right] $ be a normed weighted holomorphic ideal with the LP in 
$\mathcal{I}$. Then%
\begin{equation*}
\left[ \mathcal{A}^{-1}\mathcal{\circ I}^{\mathcal{H}_{v}^{\infty
}},\left\Vert .\right\Vert _{\mathcal{A}^{-1}\mathcal{\circ I}^{\mathcal{H}%
_{v}^{\infty }}}\right] =\left[ \left( \mathcal{A}^{-1}\circ \mathcal{I}%
\right) \mathcal{\circ H}_{v}^{\infty },\left\Vert .\right\Vert _{\left( 
\mathcal{A}^{-1}\circ \mathcal{I}\right) \mathcal{\circ H}_{v}^{\infty }}%
\right] .
\end{equation*}
\end{corollary}

Corollary \ref{com} has as a consequence the following result that shows the relationship, given by linearization, between
weighted holomorphic quotients and left-hand quotients of operator ideals
for those weighted holomorphic ideals that have the LP. This result is analogous to the linearization relationship that occurs in composition weighted holomorphic ideals and shows how both, weighted holomorphic ideals given by the composition method and weighted holomorphic ideals given by the left-hand quotients, have deep  similarities. 

\begin{corollary}
\label{4}Let $\left[ \mathcal{A},\left\Vert .\right\Vert _{\mathcal{A}}%
\right] $ be a normed operator ideal, $\left[ \mathcal{I},\left\Vert
.\right\Vert _{\mathcal{I}}\right] $ a normed operator ideal, and $\left[ 
\mathcal{I}^{\mathcal{H}_{v}^{\infty }},\left\Vert .\right\Vert _{\mathcal{I}%
^{\mathcal{H}_{v}^{\infty }}}\right] $ a normed weighted holomorphic ideal
that satisfies the LP in $\mathcal{I}$. For every $f\in \mathcal{H}%
_{v}^{\infty }(U,F)$, the following statements are equivalent:

\begin{enumerate}
\item[(i)] $f\in \mathcal{A}^{-1}\mathcal{\circ I}^{\mathcal{H}_{v}^{\infty
}}\left( U,F\right) .$

\item[(ii)] $T_{f}\in \mathcal{A}^{-1}\mathcal{\circ I}\left( \mathcal{G}%
_{v}^{\infty }(U),F\right) .$
\end{enumerate}

Moreover, in this case
\begin{equation*}
\left\Vert f\right\Vert _{\mathcal{A}^{-1}\mathcal{\circ I}^{\mathcal{H}%
_{v}^{\infty }}}=\left\Vert T_{f}\right\Vert _{\mathcal{A}^{-1}\mathcal{%
\circ I}}.
\end{equation*}
\end{corollary}

\begin{proof}
By Corollary \ref{com} 
\begin{equation*}
\left[ \mathcal{A}^{-1}\mathcal{\circ I}^{\mathcal{H}_{v}^{\infty
}},\left\Vert .\right\Vert _{\mathcal{A}^{-1}\mathcal{\circ I}^{\mathcal{H}%
_{v}^{\infty }}}\right] =\left[ \left( \mathcal{A}^{-1}\circ \mathcal{I}%
\right) \mathcal{\circ H}_{v}^{\infty },\left\Vert .\right\Vert _{\left( 
\mathcal{A}^{-1}\circ \mathcal{I}\right) \mathcal{\circ H}_{v}^{\infty }}%
\right] .
\end{equation*}
Then, using \cite[Theorem 2.7]{CaJiKe}, $f\in \mathcal{A}^{-1}\mathcal{\circ I}^{\mathcal{H}_{v}^{\infty
}}\left( U,F\right)=  (\mathcal{A}^{-1}\mathcal{\circ I})\circ {\mathcal{H}_{v}^{\infty
}}\left( U,F\right)$ 
if, and only if, $T_{f}\in \mathcal{A}^{-1}\mathcal{\circ I}\left( \mathcal{G}%
_{v}^{\infty }(U),F\right) $ and
$$
\|f\|_{\mathcal A^{-1}\circ \mathcal I^{\mathcal H_v^\infty}}=\| f\|_{(\mathcal A^{-1}\circ \mathcal I)\circ \mathcal H_v^\infty}=\|T_f\|_{\mathcal A^{-1}\circ \mathcal I}.
$$
\end{proof}

Therefore, the  map $f\to T_{f}$ defines an isometric
isomorphism from $\left( \mathcal{A}^{-1}\mathcal{\circ I}^{\mathcal{H}%
_{v}^{\infty }}\left( U,F\right) ,\left\Vert .\right\Vert _{\mathcal{A}^{-1}%
\mathcal{\circ I}^{\mathcal{H}_{v}^{\infty }}}\right) $ onto $\left( 
\mathcal{A}^{-1}\mathcal{\circ I}\left( \mathcal{G}_{v}^{\infty
}(U),F\right) ,\left\Vert .\right\Vert _{\mathcal{A}^{-1}\mathcal{\circ I}%
}\right) .$

Remember that, whenever $\mathcal I= \mathcal K_p, \mathcal K_{wp}, \mathcal K_{up},  \overline{\mathcal F}, \mathcal N^p$, all the normed weighted holomorphic ideals $\mathcal H_{v \mathcal I }^\infty$ fulfill the LP in $\mathcal  I$, and then by Proposition \ref{coincidence} or Theorem \ref{caract} $\mathcal I^{-1}\circ \mathcal H_{v\mathcal I}^\infty=\mathcal H_v^\infty$, that is the weighted holomorphic left-hand quotient  $\mathcal I^{-1}\circ \mathcal H_{v\mathcal I}^\infty$ coincides with the whole $\mathcal H_v^\infty$ as we pointed out in the former examples. 

In the next section we will see non trivial examples were  coincidence with the whole ideal $\mathcal H_v^\infty$ does not hold. 

\section{Examples of weighted holomorphic left-hand quotient ideals}\label{sec}

This section explores two instances of weighted holomorphic left-hand
quotient ideals, derived from considering the ideals of weighted holomorphic
mappings with Grothendieck and Rosenthal weighted ranges.  

Let us recall some basic definitions.
A subset $A$ of $E$ is said to be
conditionally weakly compact (or Rosenthal) if every sequence in $A$ has a
weak Cauchy subsequence. A subset $
K\subseteq E$ is called Grothendieck if, for every operator $T\in \mathcal{L}
\left( E,c_{0}\right) $, its image $T\left( K\right) $ is a relatively
weakly compact subset of $c_{0}$ (see e.g. \cite[p. 298]{Gonz2021}). 

An operator $T\in \mathcal{L}\left( E,F\right) $ is classified as separable,
Rosenthal, or Grothendieck based on whether the image of the unit ball $%
T\left( B_{E}\right) $ forms a 
separable, Rosenthal, or Grothendieck subset of $F$, respectively. The
corresponding operator ideals are denoted as  $\mathcal{S} $ for separable bounded operators, $%
\mathcal{R}$ for Rosenthal operators, and $\mathcal{O} $ for Grothendieck operators. It is well known that $\mathcal{K} \subseteq \mathcal{K}_w\subseteq \mathcal{R} $, $\mathcal{K}_w\subseteq \mathcal{O} $, and $\mathcal{K}\subseteq \mathcal{S}.$ 

An operator $T\in \mathcal{L}%
\left( E,F\right) $ is called completely continuous if it maps  weakly
convergent sequences  to  norm-convergent sequences. Let $\mathcal{V}$ denote the closed operator ideal of all
completely continuous operators from $E$ to $F$ (see \cite[1.6.2
and 4.2.5]{Pietsch1980}). 

For an extensive
analysis of these operator ideals, we refer to Pietsch's monograph \cite%
{Pietsch1980}.  It is well established that $\mathcal{O}$
constitutes a closed surjective operator ideal. For further insights into
the Grothendieck property, we direct the reader to \cite{Gonz2021}.

\begin{definition}
A mapping $f\in \mathcal{H}_{v}^{\infty }(U,F)$ is said to be Grothendieck
weighted holomorphic if  $(vf)\left( U\right) $ forms a Grothendieck
subset of $F$. We denote by $\mathcal{H}_{v\mathcal{O}}^{\infty }(U,F)$ the
space of all Grothendieck weighted holomorphic mappings from $U$ into $F.$

 A weighted holomorphic mapping $f\in \mathcal{H}_{v}^{\infty }(U,F)$ is called Rosenthal if $(vf)\left( U\right) $ is a Rosenthal subset of $F$. The class of all Rosenthal weighted holomorphic mappings is denoted by $\mathcal{H}_{v\mathcal{R}}^{\infty }$.
\end{definition}

We have seen in the previous section several examples of weighted holomorphic ideals that satisfy the
linearization property in certain operator ideals. Let us see that $\mathcal{H}_{v\mathcal{O}}^{\infty }$
 satisfies the linearization property in $\mathcal{O}$. The next theorem shows that  the mapping $f\mapsto T_{f}$ defines an isometric isomorphism
from $\left( \mathcal{H}_{v\mathcal O}^{\infty }(U,F),\left\Vert .\right\Vert
_{v}\right) $ onto $\left( \mathcal{O}(\mathcal{G}_{v}^{\infty
}(U),F),\left\Vert .\right\Vert \right) $, as well as from $\left( \mathcal{%
O\circ H}_{v}^{\infty }(U,F),\left\Vert .\right\Vert _{\mathcal{O\circ H}%
_{v}^{\infty }}\right) $ onto $\left( \mathcal{O}(\mathcal{G}_{v}^{\infty
}(U),F),\left\Vert .\right\Vert \right) .$

\begin{theorem}
\label{5}For a map $f\in \mathcal{H}_{v}^{\infty }(U,F)$, the following are
equivalent:

\begin{enumerate}
\item[(i)] $f\in \mathcal{H}_{v\mathcal{O}}^{\infty }(U,F).$

\item[(ii)] $T_{f}\in \mathcal{O}(\mathcal{G}_{v}^{\infty }(U),F).$

\item[(iii)] $f\in \mathcal{O\circ H}_{v}^{\infty }(U,F).$
\end{enumerate}

In this case, we have: 
\begin{equation*}
\left\Vert f\right\Vert _{v}=\left\Vert T_{f}\right\Vert =\left\Vert
f\right\Vert _{\mathcal{O\circ H}_{v}^{\infty }}.
\end{equation*}%

\end{theorem}

\begin{proof}
$($i$)\Rightarrow ($ii$)$: If $f\in \mathcal{H}_{v\mathcal{O}}^{\infty
}(U,F) $, for each $x\in U$ we have $T_{f}\left( \Delta _{v}\left( x\right) \right) =f\left(
x\right)$, and then $(v(T_f\circ \Delta_v)(U))$ is a Grothendieck set in $F$. Notice that the norm-closed absolutely
convex hull (denoted as $\overline{\text{abco}}$) of a Grothendieck set is itself Grothendieck due to the
norm-closed absolutely convex hull of a relatively weakly compact set is
relatively weakly compact. Therefore, $\overline{\text{abco}}\left(v
T_{f}\circ\Delta _{v}\left( U \right) \right) $ is Grothendieck in $%
F$. Since 
\begin{equation*}
T_{f}\left( B_{\mathcal{G}_{v}^{\infty }(U)}\right) =T_{f}\left( \overline{%
\text{abco}}\left( \Delta _{v}\left( U\right) \right) \right) \subseteq 
\overline{\text{abco}}\left( T_{f}\left( \Delta _{v}\left( U\right) \right)
\right)
\end{equation*}
it follows that $T_{f}$ is a Grothendieck linear operator.

$($ii$)\Rightarrow ($i$)$: If $T_{f}\in \mathcal{O}(\mathcal{G}_{v}^{\infty
}(U),F)$, then $T_{f}\left( B_{\mathcal{G}_{v}^{\infty }(U)}\right) $ is a
Grothendieck subset of $F$. Since $(vf)(U)\subseteq T_f(B_{%
\mathcal{G}_{v}^{\infty }(U)})$, it follows that $(vf)\left( U\right)
 $ is a Grothendieck set in $F$.

$($ii$)\Leftrightarrow ($iii$)$: This follows as a direct application of 
\cite[Theorem 2.7]{CaJiKe}.

The equality of the norms follows easily from the linearization properties and \cite[Theorem 2.7]%
{CaJiKe}.
\end{proof}

As a consequence of Proposition \ref{coincidence} we get:
\begin{corollary}
\begin{equation*}
\left[ \mathcal{O}^{-1}\circ \mathcal{H}_{v\mathcal O}^{\infty
},\left\Vert .\right\Vert _{\mathcal{O}^{-1}\mathcal{\circ I}^{\mathcal{H}%
_{v\mathcal O}^{\infty }}}\right] =\left[ \mathcal{H}_{v}^{\infty },\left\Vert
.\right\Vert _{v}\right]
\end{equation*}%
\end{corollary}

As a consequence of Theorem \ref{5} and \cite[Corollary 2.5]{CaJiRu} we get that $\left[ \mathcal{H}_{v\mathcal{O}}^{\infty },\left\Vert
.\right\Vert _{v}\right] $ is a closed weighted holomorphic ideal.

The next result shows that Grothendieck weighted holomorphic mappings are characterized as those whose composition with separable bounded linear operators gives rise to weakly compact weighted holomorphic mappings.

\begin{theorem}
\label{7}$\left[ \mathcal{H}_{v\mathcal{O}}^{\infty },\left\Vert
.\right\Vert _{v}\right] =\left[ \mathcal S^{-1}\circ \mathcal{H}_{v\mathcal{K}_w%
}^{\infty },\left\Vert .\right\Vert _{\mathcal S^{-1}\circ \mathcal{H}_{v\mathcal{K}_w%
}^{\infty }}\right] .$
\end{theorem}

\begin{proof}
Let $f\in \mathcal{H}_{v}^{\infty }(U,F)$. Considering Theorem \ref{5}, \cite%
[Proposition 3.2.6]{Pietsch1980}
and Corollary \ref{4}, we get the following%
\begin{eqnarray*}
f &\in &\mathcal{H}_{v\mathcal{O}}^{\infty }(U,F)\Leftrightarrow T_{f}\in 
\mathcal{O}(\mathcal{G}_{v}^{\infty }(U),F) \\
&\Leftrightarrow &A\circ T_{f}\in \mathcal{K}_w\left( \mathcal{G}_{v}^{\infty
}(U),G\right) ,\qquad \forall A\in \mathcal S\left( F,G\right) \\
&\Leftrightarrow &T_{f}\in \mathcal S^{-1}\circ \mathcal{K}_w\left( \mathcal{G}%
_{v}^{\infty }(U),F\right) \\
&\Leftrightarrow &f\in \mathcal S^{-1}\circ \mathcal{H}_{v\mathcal{K}_w}^{\infty
}\left( U,F\right).
\end{eqnarray*}

with the equalities of the norms $\left\Vert f\right\Vert _{v}=\left\Vert T_{f}\right\Vert
=\left\Vert T_{f}\right\Vert _{\mathcal S^{-1}\circ \mathcal{K}_w}=\left\Vert
f\right\Vert _{\mathcal S^{-1}\circ \mathcal{H}_{v\mathcal{K}_w}^{\infty }}.$
\end{proof}

Our last result describes the ideal of weighted holomorphic mappings with a
Rosenthal weighted range as a
weighted holomorphic left-hand quotient ideal generated by the operator
ideal $\mathcal{V}$ and the weighted holomorphic ideal $\mathcal{H}_{v%
\mathcal{K}}^{\infty }.$ 

We need to recall first that a linear operator $Q:F\to G$ between Banach spaces is $r$-integral ($1\leq r\leq \infty$)  if there are a probability measure $\mu$ and bounded linear operators $a:L_r(\mu)\to G^{**}$ and $b:F\to L_\infty(\mu)$ such that $k_G\circ Q=a\circ i_r\circ b$, where $i_r:L_\infty (\mu)\to L_r(\mu)$ is the formal identity and $k_G:G\to G^{**}$ is the canonical isometric embedding.

\begin{theorem}
$\left[ \mathcal{H}_{v\mathcal{R}}^{\infty },\left\Vert .\right\Vert _{v}%
\right] =\left[ \mathcal{V}^{-1}\circ \mathcal{H}_{v\mathcal{K}}^{\infty
},\left\Vert .\right\Vert _{\mathcal{V}^{-1}\circ \mathcal{H}_{v\mathcal{K}%
}^{\infty }}\right] .$
\end{theorem}

\begin{proof}
Given $f\in \mathcal{H}_{v\mathcal R}^{\infty }(U,F)$,  every sequence $((vf)(x_n))_n$ with $(x_n)_n$ in $U$ admits a weakly Cauchy subsequence $((vf)(x_{n_k}))_k$. Then, for every Banach space $G$ and every linear operator $T\in\mathcal V(F,G)$ the sequence $(T(vf)(x_{n_k}))_k$ is norm convergent (\cite[Proposition 1.6.3]{Pietsch1980}). 

Besides, 
$$
\|v(x)(Tf(x))\|=\|T(v(x)f(x))\|\leq \|T\| \ \|f\|_v.
$$
Hence, $Tf\in \mathcal H_{v\mathcal K}^\infty(U,G)$. Thus, $f\in \mathcal V^{-1}\circ \mathcal H_{v\mathcal K}^\infty(U,F)$.

Conversely, let $f\in \mathcal V^{-1} \circ \mathcal H_{v\mathcal K}^\infty(U,F)$ and let us suppose that $f\notin \mathcal{H}_{v\mathcal R}^{\infty }(U,F)$. Then, there exists a sequence $(x_n)_n$ in $U$ such that $((vf)(x_n))_n$ does not contain any subsequence which is weakly Cauchy. As the sequence $((vf)(x_n))_n$ is bounded, by the well known Rosenthal theorem we can find a subsequence $((vf)(x_{n_k}))_k$ and a linear  injection $J:\ell_1\to E$ such that $(vf)(x_{n_k})=J(e_k)$ for all $k$, where $(e_k)_k$  is the canonical basis of $\ell_1$. According to the proof of \cite[Theorem 28.5.8]{Pietsch1980} if $Q_0:\ell_1\to \ell_2$ is any linear surjection then, $Q_0$ is an $r$-integral linear operator that can be extended to an $r$-integral linear operator $Q:F\to\ell_2$ such that $Q_0=Q\circ J$. Then $Q_0(e_k)=Q\circ (vf)(x_{n_k})_k$ for all $k$.

As any $r$-integral linear operator is completely continuous, we have that $Q\in \mathcal V$ and then, by assumption on $f$,  $Q\circ f\in \mathcal H_{v\mathcal K}^\infty(U,\ell_2)$. Since
$$
Q_0(B_{\ell_1})\subset \overline{abco}(Q_0(e_k))=\overline{abco}(Q\circ J(e_k))=\overline{abco}(Q\circ(vf)(x_{n_k}))\subset \overline{abco}(Q\circ(vf)(U))
$$
we conclude that $Q_0$ is a compact linear operator, which is a contradiction.

Moreover, 
\begin{eqnarray*}
\|f\|_{\mathcal V^{-1}\circ \mathcal H_{v\mathcal K}^\infty}&=& \sup\{\|T\circ f\|_v:T\in \mathcal V(F,G), \|T\|\leq 1, G \mbox{ Banach }\}\\
&\leq & \sup\{\|f\|_v\|T\|: T\in \mathcal V(F,G), \|T\|\leq 1, G \mbox{ Banach }\}\\
&\leq & \|f\|_v
\end{eqnarray*}
and by Proposition \ref{coincidence} we get the equality of the norms.

\end{proof}


\begin{thebibliography}{99}

\bibitem{BaDe1}  D. Baweja; M. Gupta, Characterizations of approximation properties defined by operator ideals in weighted Banach spaces of holomorphic functions. Adv. Oper. Theory 7 (2022), no. 3, Paper No. 39, 11 pp.

\bibitem{BlaGaMi} O. Blasco; P. Galindo; M. Lindstr\"om; A. Miralles, Interpolating sequences for weighted spaces of analytic functions on the unit ball of a Hilbert space. Rev. Mat. Complut. 32 (2019), no. 1, 115--139.

\bibitem{BoPeRu} G. Botelho; D. Pellegrino; P. Rueda, On composition ideals of multilinear mappings and homogeneous polynomials. Publ. Res. Inst. Math. Sci. 43 (2007), no. 4, 1139--1155.

\bibitem{CaJiKe} M. G. Cabrera-Padilla; A. Jim\'enez-Vargas; A. Keten Çopur,  Weighted holomorphic mappings associated with p-compact type sets. Bull. Malays. Math. Sci. Soc. 48 (2025), no. 2, Paper No. 32, 21 pp.

\bibitem{CaJiRu} M. G. Cabrera-Padilla; A. Jim\'enez-Vargas; D. Ruiz-Casternado, On composition ideals and dual ideals of bounded holomorphic mappings. Results Math. 78 (2023), no. 3, Paper No. 103, 21 pp.

\bibitem{CaDe} B. Carl; A. Defant, Tensor products and Grothendieck type inequalities of operators in $L_p$-spaces, Trans. Amer. Math. Soc.
331 (1992), no. 1, 55--76.

\bibitem{CaNa} R. M. Causey; K. V. Navoyan, $\xi$-completely continuous operators and $\xi$-Schur Banach spaces, J. Funct. Anal. 276 (2019), no.
7, 2052--2102.


\bibitem{GaMaRu}  D. Garc\'ia; M. Maestre; P. Rueda, Weighted spaces of holomorphic functions on Banach spaces. Studia Math. 138 (2000), no. 1, 1--24.


\bibitem{Gonz2021} M. Gonz\'{a}lez and T. Kania.: Grothendieck spaces: the
landscape and perspectives, Japan. J. Math \textbf{16} (2021), 247--313.

\bibitem{BaDe} M. Gupta; D. Baweja, The bounded approximation property for the weighted spaces of holomorphic mappings on Banach spaces. Glasg. Math. J. 60 (2018), no. 2, 307--320.

\bibitem{BaDe3} M. Gupta; D. Baweja,  Weighted spaces of holomorphic functions on Banach spaces and the approximation property. Extracta Math. 31 (2016), no. 2, 123--144.

\bibitem{BeBuRu} A. Belacel; A. Bougoutaia; P. Rueda,  Summability of multilinear operators and their linearizations on Banach lattices. J. Math. Anal. Appl. 527 (2023), no. 2, Paper No. 127459, 18 pp. 

\bibitem{BuBu} Q. Bu, G. Buskes Polynomials on Banach lattices and positive tensor products J. Math. Anal. Appl., 388 (2012), pp. 845-862

\bibitem{JiRaVi}  A. Jiménez-Vargas; M. I. Ramírez; M. Villegas-Vallecillos, The Bishop-Phelps-Bollobás property for weighted holomorphic mappings. Results Math. 79 (2024), no. 4, Paper No. 155, 25 pp.

\bibitem{JimAysRuiz24} A. Jim{\'{e}}nez-Vargas; D. Ruiz-Casternado, On
quotients of ideals of bounded holomorphic maps. Filomat 38:29, (2024), 10123--10132.

\bibitem{JoLiO} W. B. Johnson; R. Lillemets; E. Oja, Representing completely continuous operators through weakly $\infty$-compact operators,
Bull. London Math. Soc. 48 (2016), no. 3, 452--456.


\bibitem{Jo} E. Jord\'a, Weighted vector-valued holomorphic functions on Banach spaces. Abstr. Appl. Anal. 2013, Art. ID 501592, 9 pp.


\bibitem{Ki} J. M. Kim, The ideal of weakly $p$-compact operators and its approximation property for Banach spaces, Ann. Fenn. Math. 45
(2020), no. 2, 863--876.

\bibitem{Mu} J. Mujica,  Linearization of bounded holomorphic mappings on Banach spaces. Trans. Amer. Math. Soc. 324 (1991), no. 2, 867--887.

\bibitem{PeRuSa} D. Pellegrino; P. Rueda; E. A. S\'anchez-P\'erez,  Surveying the spirit of absolute summability on multilinear operators and homogeneous polynomials. Rev. R. Acad. Cienc. Exactas Fís. Nat. Ser. A Mat. RACSAM 110 (2016), no. 1, 285--302.

\bibitem{Pietsch1980} A. Pietsch, Operator Ideals, North-Holland
Mathematical Library, vol. 20. North-Holland Publishing Co., Amsterdam
(1980).

\bibitem{Puhl1977} J. Puhl, Quotienten von Operatorenidealen, Math. Nachr. 
\textbf{79} (1977), 131--144.

\bibitem{Ru} P. Rueda,  On the Banach-Dieudonn\'e theorem for spaces of holomorphic functions. Quaestiones Math. 19 (1996), no. 1--2, 341--352.

\end{thebibliography}
\end{document}